\documentclass[11pt,reqno]{amsart}
\usepackage[utf8]{inputenc}
\usepackage[margin=1in]{geometry}
\usepackage{mathpazo}

\usepackage{amsmath}
\usepackage{graphicx}
\usepackage[colorlinks=true, allcolors=blue]{hyperref}
\usepackage{xcolor}
\usepackage{tikz,tikz-cd}
\usetikzlibrary{quotes}

\newcommand{\calM}{\mathcal{M}}
\newcommand{\calC}{\mathcal{C}}

\newcommand{\C}{\mathbb C}

\newcommand{\Z}{\mathbb Z}

\newtheorem{thm}{Theorem}[section]

\newtheorem{lem}[thm]{Lemma}

\newtheorem{cor}[thm]{Corollary}

\newtheorem{prop}[thm]{Proposition}

\theoremstyle{definition} 

\title{Complete Families of Curves in the Moduli Space of Genus g Curves}

\author[Stewart]{Chloe Stewart}
\address{Department of Mathematics, Colorado State University\\ 
Fort Collins, CO 80523, USA}
\email{chloe.stewart@colostate.edu}

\begin{document}

\begin{abstract}
Let $\calM_g$ be the moduli space of smooth curves of genus $g$.  The image of a non-constant morphism from a curve $T$ to $\calM_g$ is a curve in $\calM_g$.  By work of Gonz\'alez D\'iez and Harvey, for every integer $g \geq 3$, there exists a complete curve in $\calM_g$.  Here we generalize the construction to produce new complete curves in $\calM_g$.  We also find a formula for the genus of each curve $T$ using Galois theory for function fields. 

MSC20 classifications: primary 11S15, 14H10, 14H30; secondary 11R32, 14D20, 14H37 \\


\end{abstract}

\maketitle

\section{Introduction}\label{sect:intro}

The points of $\calM_g$ represent isomorphism classes of smooth, projective, and connected genus $g$ curves.  A curve in $\calM_g$ is the image of a non-constant morphism $T \longrightarrow \calM_g$, where T is a curve.  Let $\calC$ be the image of $T$ in $\calM_g$.  Then $\calC$ is projective (or complete) if $T$ is projective.  To define a complete curve in $\calM_g$ we can define a non-isotrivial family of genus $g$ curves over a projective curve such that every fiber is smooth.

If $g \geq 3$ then many complete curves exist in $\calM_g$.  One can abstractly construct complete curves in $\calM_g$ by intersecting hypersurfaces \cite{Harris1983}. The first examples of complete curves were constructed by Gonz\'alez D\'iez and Harvey \cite{GonDiezHarvey1991}, \cite{GonDiezHarveyII}. The first paper contains an explicit construction for a complete curve for each $g>3$.  The second paper handles the $g=3$ case, but is not as explicit and uses a different strategy. Later, Zaal returned to the $g=3$ case and found a more explicit construction \cite{Zaal1995}.

For $g>3$, the construction in \cite{GonDiezHarvey1991} amounts to specifying a family of genus $g$ curves.  The family is given by varying the branch points of a double cover of a fixed genus $2$ curve.  In this paper, we generalize this work by generating families where each curve is a double cover of a fixed genus $\overline{g}\geq 2$ curve.  Let $\calM_{g,\overline{g}} \subset \calM_g$ denote the sublocus whose points represent isomorphism classes of smooth genus $g$ curves which are double covers of a genus $\overline{g}$ curve.  

\begin{thm} [\cite{GonDiezHarvey1991}]
\label{GDHConstruction}
For every $g > 3$, there exists a complete curve in $\calM_{g,2}$.
\end{thm}

We prove the following theorem in Section \ref{sect:genconstruction}.  We work over an algebraically closed field denoted by $k$ where $\mathrm{char}(k) \not = 2$.

\begin{thm} \label{generalconst}
For every $\overline{g} \geq 2$ and $g> 2\overline{g}-1$, there exists a complete curve $\calC$ in $\calM_{g,\overline{g}}$ over $k$.
\end{thm}
 
In Sections \ref{sect:genusT} and \ref{sect:altconst}, we investigate the geometry of the curves arising from this construction.  After making a choice of the base curve and imposing a simple condition, we calculate the genus of every $T$ using the Riemann Hurwitz formula.

This work is connected to an important open question pertaining to $\calM_g$.  The maximal dimension for a complete subvariety in $\calM_g$ for $g \geq 4$ is unknown.  This question is interesting because the maximal dimension of a complete subvariety in a moduli space is a measure of how close the moduli space is to being projective or affine.  For $g= 1,2$ the only complete subvarieties are points, and $ \calM_1$ and $\calM_2$ are affine.  In $\calM_3$, the maximal dimension of a complete subvariety is $1$.  However, it is unknown if a $2$-dimensional complete subvariety exists in $\calM_4$.  There are several results bounding the maximal dimension of a complete subvariety of $\calM_g$.  Diaz proved, in $\mathrm{char}(k)=0$, that if $Z \subset \calM_g$ is a complete subvariety then $\mathrm{dim}(Z)\leq g-2$ \cite{Diaz1984}.  Later Looijenga proved this upper bound in general characteristic \cite{Looijenga1995}.  Zaal proved a lower bound in Section 2.2 of \cite{Zaal2005}: he proves the maximum dimension of a complete subvariety in $\calM_g$ must be greater than $\log_2(g)-2$. More recent work by Choi improves this lower bound in positive characteristic \cite{Choi2023}.  In this paper we gain a greater understanding of explicit constructions of complete curves in $\calM_g$ which may aid in constructing higher dimensional complete subvarieties.  

There is also related ongoing work defining (non-complete) families of curves with automorphisms.  As in this paper, these families define subvarieties of $\calM_g$.  For the work most similar to this paper see \cite{Hidalgo2024}, \cite{Shaska2006}, and \cite{Shaska2004}.

 \subsection*{Acknowledgments}  I would like to thank my advisor Dr.\ Rachel Pries for her expertise and guidance on this project.

\section{Complete Curve Construction} \label{sect:genconstruction}

A curve over $k$ is a smooth, projective, connected, variety of dimension $1$ over $k$.  In this section, we consider all curves to be defined over $k = \overline{k}$ where $\mathrm{char}(k) \not = 2$.  Let $X$ be a curve over $k$ with genus $\overline{g}$.  Suppose $\rho: X \longrightarrow E$ is a cover of an elliptic curve.  Given $r \geq 2$, let $\Delta_X$ be the strong diagonal of $X^r$.

\begin{lem} \label{constZ}
    Given $\overline{g}$ and $g$ such that $g> 2\overline{g}-1$, let $r=2(g+1-2\overline{g})$.  Suppose $X$ is a curve of genus $\overline{g}$ and $\tau = (x_1,\ldots,x_r ) \in X^r-\Delta_X$.  Then there exists a degree $2$ cover $ \phi: Z \longrightarrow X$ branched at $\tau$ where $Z$ is smooth and has genus $g$.
\end{lem}

\begin{proof}
Recall $X$ is a smooth projective curve over $k$.  Since $k$ is algebraically closed and $\mathrm{char}(k)$ is relatively prime to $\# \Z / 2 \Z$ it suffices to prove the result over $\C$.

Then since $\tau$ is a finite set of points, isomorphism classes of double covers with branch locus in $\tau$ are in bijection with homomorphisms from the topological fundamental group $\pi_1(X-\tau)$ to $S_2$.

The fundamental group of $X-\tau$ is 
$$\pi_1(X-\tau) = \langle l_1,l_2,\ldots,l_r,\alpha_1,\ldots,\alpha_{\overline{g}},\beta_1,\ldots,\beta_{\overline{g}} \rangle/(l_1l_2\ldots l_r)\prod[\alpha_i,\beta_i].$$ Here $l_i$ is a generator for the loop around $x_i$ and $\alpha_i$ and $\beta_i$ are the standard generators of $\pi_1(X)$.  

Define a map from $\pi_1(X-\tau)$ to $S_2 = \Z/2\Z$.  Let the map take each $l_i$ to the non-trivial element of $\Z/2\Z$ and the generators $\alpha_i$ and $\beta_i$ can be mapped anywhere.  Choose one mapping without loss of generality.  Note, $r$ is even so the map is a homomorphism.

Let the corresponding cover be $\phi: Z \longrightarrow X$.  Since every $l_i$ is mapped to a non-trivial element, the cover $\phi: Z \longrightarrow X$ is branched at every point in $\tau$.  We use the Riemann-Hurwitz formula to verify $Z$ has genus $g$.  Let $e_z$ be the ramification index for $z \in Z$.  There are $r$ branch points, each with ramification index $2$, so $\sum (e_z - 1) = r$. So $ G_Z = g $, because

\begin{center}
$ 2G_Z -2 = (d)(2\overline{g} - 2) + \sum (e_z -1). $
\end{center}
\end{proof}

\begin{lem} \label{eccover}

For every $\overline{g} \geq 2$, there exists a cover $\rho: X \longrightarrow E$ of smooth curves such that $X$ has genus $\overline{g}$ and $E$ is an elliptic curve.

\end{lem}

\begin{proof}
    We proceed similarly as above.  Let $d$ be an integer such that $d \geq 1$.  Let $\Delta_E$ be the strong diagonal of $E^{2d}$.  Define $B = ( P_1, \ldots ,P_{2d} ) \in E^{2d} - \Delta_E$.  The fundamental group of $E-B$ is
    $$\pi_1(E-B) = \langle l_1,l_2, \ldots ,l_{2d},\alpha,\beta \rangle/(l_1l_2 \ldots l_{2d})[\alpha,\beta].$$
    Here $l_i$ is the loop around $P_i$ and $\alpha$ and $\beta$ are the generators of $\pi_1(E)$.

    As before, define $\psi: \pi_1(E-B) \longrightarrow \Z/2\Z$ such that $\psi (l_i) = 1$ and $\psi (\alpha) = \psi(\beta) = 0$.  Then the corresponding degree two cover $\rho: X \longrightarrow E$ of smooth curves is branched at the points of $B$.  

    Using the Riemann Hurwitz formula we find the genus of $X$.  Similarly to the previous proof, $\sum (e_x -1 )= 2d$.  We obtain $ G_X = d+1 $.

Since $d \geq 1$, by varying the number of branch points, we can produce at least one $\rho: X \longrightarrow E$ where $X$ has genus $\overline{g}$ for all $\overline{g} \geq 2$.
\end{proof}

For a similar construction and a discussion of the Hurwitz space which parametrizes degree $d$ genus $\overline{g}$ connected covers of an elliptic curve, see \cite{Chen2010}.

Now we prove Theorem \ref{generalconst}.

\begin{proof}[Proof of Theorem \ref{generalconst}]
By Lemma \ref{eccover}, for every $\overline{g} \geq 2$, there exists a cover $\rho: X \longrightarrow E$ such that the genus of $X$ is $\overline{g}$. Then as in Lemma \ref{constZ}, let $r=2(g+1-2\overline{g})$.  Define $+$ to be the group action on the points of $E$.  Fix distinct points $Q_1=\mathcal{O}, Q_2, \ldots , Q_r \in E$ and let $x_i \in X$.  Define the following cover:
\begin{equation} \label{ToverWgen}
\begin{tikzcd}
 T:= \{ (x_1, \ldots , x_r) \mid \rho (x_j) = \rho(x_1) + Q_j \} \dar{\rho^r} \\
 W:= \{ (P, P + Q_2, \ldots  P + Q_r) \mid P \in E \},
 \end{tikzcd}
\end{equation}

where $\rho^r((x_1,\ldots ,x_r))= (\rho(x_1),\ldots ,\rho(x_r))$. Note that $W \cong E$ because $Q_1,\ldots ,Q_r$ are fixed.  So, $W$ is a smooth projective curve.  Then $T$ is a projective curve because it is a finite cover of a projective curve.  

For all $(P,P+Q_1,\ldots ,P+Q_r) \in W$ the entries are distinct.  Then $T \subset X^r - \Delta_{X}$ because if $\rho^{-1}(P + Q_j)$ intersects $ \rho^{-1}(P + Q_l)$ for some $l \not = j$ then $P+Q_j = P+ Q_l$, which is a contradiction.  In other words, $\tau \in T$ has no repeated entries.

Apply Lemma \ref{constZ} to every $\tau \in T$ to obtain a smooth genus $g$ curve $Z$ branched at $\tau$ for every choice of $\tau \in T$.  Recall the lemma holds for $g$ and $\overline{g}$ such that $g+1-2\overline{g}>0$.  This family of curves depends only on the parameter $P$.

To verify this family of curves is non-isotrivial, we note that there are finitely many involutions in the automorphism group of each curve $Z$. This means there are finitely many double covers from $Z$ to a curve of genus $\overline{g}$ if $\overline{g}\geq2$.  So, at most finitely many $Z$'s are isomorphic to each other and the family is not isotrivial.

We have defined a complete $1$-parameter family of smooth genus $g$ curves which is not isotrivial where each curve is a double cover of a genus $\overline{g}$ curve when $\overline{g} \geq 2$ and $ g \geq 2 \overline{g} -1$.  Then the image of $T \longrightarrow \calM_{g,\overline{g}}$, $\mathcal{C}$ is a complete curve.

\end{proof}

\section{An Explicit Family of Complete Curves} \label{sect:Xmconst}
In this and subsequent sections, suppose $\mathrm{char}(k)=0$.  Recall that the points of $\calM_{g,\overline{g}} \subset \calM_g$ represent isomorphism classes of genus $g$ curves which are degree $2$ covers of a genus $\overline{g}$ curve.

We can specialize the construction in Section \ref{sect:genconstruction} to a particular choice of cover $\rho: X \longrightarrow E$ to obtain an explicit complete curve in $\calM_{g,\overline{g}}$.  Let $E$ be the elliptic curve with affine equation $y_1^2=x_1^3-1$.  For $m>1$, let $X_m$ be the smooth projective curve with affine equation 
\begin{equation} \label{Xm}
y_2^2=x_2^{3m}-1.
\end{equation}
Then there is a cover of curves $\rho_m: X_m \longrightarrow E$, which takes $(x_2,y_2) \mapsto (x_2^m,y_2)$. 

\begin{lem} \label{ghatmgenusandgal}
        For $m>1$, the cover defined by $\rho_m: X_m \longrightarrow E$ is a Galois cover with Galois group $\mathrm{Gal}(X_m/E) \cong \Z/m\Z$ and the curve $X_m$ has genus  
\begin{equation} \label{gbareqn} \overline{g}= \begin{cases} 
      (3m-1)/(2) & \text{ if } m \text{ odd}. \\
      (3m/2)-1 & \text{ if } m \text{ even}.
   \end{cases}
\end{equation}
\end{lem}

\begin{proof}
The genus follows from the genus formula for hyperelliptic curves.
    
A cover is Galois if and only if the corresponding function field extension is Galois.  The Galois group of the cover matches the Galois group of the associated function field extension.  The function fields arising from $E$ and $X_m$ are $F := \mathrm{Frac}(\C[x_1,y_1]/\langle y_1^2-(x_1^3-1) \rangle)$ and $K : =\mathrm{Frac}(\C[x_2,y_2]/\langle y_2^2 -(x_2^{3m}-1)\rangle )$ respectively.  The field extension $K/F$ is given by adjoining the $m$th root of $x_1$, so the Galois group is $\Z/m\Z$.
\end{proof}

When $m$ is even, $m=2d+2$ for $d \geq 0$.  This implies $\overline{g}=(3m/2)-1 = 3d+2 \equiv 2 \bmod 3$.  Similarly, when $m$ is odd, $\overline{g} \equiv 1 \bmod 3$.

Note that for any abelian cover the ramification indices are equal for every ramification point in the fiber above a branch point.  The cover $\rho_m: X_m \longrightarrow E$ is abelian by Lemma \ref{ghatmgenusandgal}, so we will refer to the ramification index of a branch point (i.e., the ramification index of any point in the fiber).

\begin{cor} \label{rho_mCC}
Fix $m$ and let $\overline{g}$ be as in equation \eqref{gbareqn}.  The choice of $\rho: X_m \longrightarrow E$ defined in Lemma \ref{ghatmgenusandgal} together with a choice of points $Q_2,...,Q_r \in E$ produces an explicit complete curve in $\calM_{g,\overline{g}}$ by defining a family of covers $\phi: Z_\tau \longrightarrow X_m$ branched at $\tau$ as in \eqref{ToverWgen} for all $\overline{g} \geq 2$ such that $\overline{g} \equiv 1, 2 \bmod 3$ and $g > 2\overline{g}-1$.
\end{cor}

\begin{proof}
The points $(0,\pm i)$ are branch points of $\rho_m$ and have ramification index $m-1$.  For values of $m>2$, the point at infinity is also a ramification point.

 Case 1. The degree $m$ is even.  Then $X_m$ has two points at infinity and the genus of $X_m$ is $\overline{g} = (3m/2)-1$.  The result follows immediately from Theorem \ref{generalconst} and Lemma \ref{ghatmgenusandgal}.  Notice, we produce a complete curve in $\calM_{g,\overline{g}}$ for all $g > 3m - 3$.
 
Case 2. The degree $m$ is odd.  Then $X_m$ has one point at infinity and the genus of $X_m$ is $\overline{g} = (3m-1)/2$.  Then once again we apply Theorem \ref{generalconst} to produce a complete curve in $\calM_{g,\overline{g}}$ for all $g > 3m - 2$.
\end{proof}

  \begin{lem} \label{ramindexXm}
    Let $e$ be the ramification index of be a branch point $L \in E$ of the cover $\rho: X \longrightarrow E$.  The value of $e$ is
    $$e = \begin{cases}
    m, & \text{ if } L = (0,i) \text{ or } (0,-i); \\ 
    m, & \text{ if } \overline{g} > 2, \hspace{3pt} \overline{g} \equiv 1 \bmod 3,   \text{ and }  L = \mathcal{O}; \\
    m/2, & \text{ if } \overline{g} > 2, \hspace{3pt} \overline{g} \equiv 2 \bmod 3,   \text{ and }  L = \mathcal{O}. \\
    \end{cases}$$
\end{lem}

\begin{proof}
    Use the definition of the map to conclude that $L \in \{ (0,i),(0,-i), \mathcal{O}\}$.

    Case 1. If $L = (0,i)$ or $(0,-i)$, then $\rho_m^{-1}(0,i) = \{(0,i)\}$ and $\rho_m^{-1}(0,-i) = \{(0,-i)\}$.  The cover is totally ramified so $e = m$.  

    Case 2. If $\overline{g} \equiv 1 \bmod 3,$ and $L = \mathcal{O}$, then $X_m$ with genus $\overline{g}$ has one point at infinity. So, the pre-image of $\mathcal{O}$ is a single point and $e = m$ as in case 1.

    Case 3. If $\overline{g} \equiv 2 \bmod 3,$ and $L = \mathcal{O}$, then $X_m$ with genus $\overline{g}$ has two points at infinity.  Since $\rho_m$ is Galois, each of the points at infinity will have ramification index $m/2$.  Note that $m$ is always even in this case.
\end{proof}

\section{The Geometry of $\calC$} \label{sect:genusT}
  Once we fix a cover $\rho: X \longrightarrow E$ to build a curve in $\calM_g$ as above, we can draw conclusions about the geometry of $\calC$.  In this section, we let $\rho_m: X_m \longrightarrow E$, as defined in Section \ref{sect:Xmconst}, be the choice of cover for a fixed $m>1$.  Define $T_m$ as in \eqref{ToverWgen} and $\rho_m^r$ as the associated covering map.
\begin{equation} \label{TmmoverW}
\begin{tikzcd}
 T_m:= \{ (x_1, \ldots , x_r) \mid \rho_m (x_j) = \rho_m(x_1) + Q_j \} \dar{\rho_m^r} \\
 W:= \{ (P, P + Q_2, \ldots  P + Q_r) \mid P \in E \},
 \end{tikzcd}
\end{equation}
where $x_i \in X_m$ and $\rho_m^r((x_1,\ldots x_r))=(\rho_m(x_1),\ldots ,\rho_m(x_r))$.

The main result of this section is a formula for the genus of $T_m$ for all $m>1$.  To prove the result, we use the Riemann-Hurwitz Formula on the cover $\rho^r_m: T_m \longrightarrow W$.  The most important step is to show that $T_m$ is connected.

\subsection{Structure of the cover $T_m \longrightarrow W$}

Recall from the proof of Theorem \ref{generalconst} that $r=2(g+1-2\overline{g})$ and that $W$ is a smooth projective curve whose genus is $1$.
Also note, we can rewrite Lemma \ref{ghatmgenusandgal} so that for $X_m$ with genus $\overline{g}$, 
\begin{equation} m= \begin{cases} 
      (2/3)(\overline{g}+1) & \text{ if } \overline{g} \equiv 2 \bmod 3. \\
      (2\overline{g}+1)/3 & \text{ if } \overline{g} \equiv 1 \bmod 3.
   \end{cases}
 \label{mequals} \end{equation}

\begin{lem} \label{TWdeggal}
    The cover $\rho_m^r: T_m \longrightarrow W$ has degree $m^r$ and is Galois with $\mathrm{Gal}(T_m/W) \equiv (\Z/m\Z)^r$.
\end{lem}

\begin{proof}

Recall $\rho_m^r ((u_1,v_2),\ldots ,(u_r,v_r)) = ((u_1^m,v_2),\ldots ,(u_r^m,v_r))$.  A pre-image is
 $$(\rho_m^r)^{-1}((u_1,v_2),\ldots ,(u_r,v_r))= \{ ((\zeta_m^{l_0} \sqrt[m]{u_1},v_1),\ldots ,(\zeta_m^{l_r} \sqrt[m]{u_r},v_r)) \mid 0 < l_j \leq m  \}.$$  A non-branch point has $m^r$ pre-images.  Note that by equation \eqref{mequals} $m^r$ depends on $\overline{g}$.

Let $e_j$ be the $j$th standard basis element of $(\Z/m\Z)^r$.  The element $e_j$ acts on the fiber by multiplying the $j$th coordinate by $\zeta_m$, taking $(\zeta_m^{l_j}\sqrt[m]{u_j},v_j)$ to $(\zeta_m^{l_j+1}\sqrt[m]{u_j},v_j)$.  So, $\mathrm{Gal}(T_m/W) \cong (\Z/m\Z)^r$.  This action is transitive on the fiber, so the cover is Galois.
    
\end{proof}

Now that we've established basic facts about the structure of the cover, we will examine the ramification.

\subsection{Ramification of the cover $T_m \longrightarrow W$}

To simplify the ramification of $\rho_m^r: T_m \longrightarrow W$, we impose a simple condition on $Q_1,\ldots ,Q_r$.  The condition is harmless to the construction because each $Q_j$ is fixed.

\begin{lem} \label{TWbranch}
    Suppose $Q_j - Q_l \not = (0,\pm i)$ for all $1 \leq j,l \leq r$. Let $B$ be the set of branch points of $\rho^r_m: T_m \longrightarrow W$. A point $\omega \in W$ is in $B$ if and only if exactly one entry of $\omega$ is a branch point of $\rho_m: X_m \longrightarrow E$.
\end{lem}

\begin{proof}
    Any point in $W$ with at least one entry that is a branch point of $\rho_m$ is a branch point of $\rho_m^r$.
    
    We claim that by the condition $Q_j - Q_l \not = (0,\pm i)$ for all $1 \leq j,l \leq r$, at most one entry of $\omega$ is a branch point of $\rho_m$.  Recall the branch points of $\rho_m$ are $(0,i)$, $(0,-i)$, and possibly the point at infinity.  Note that using the addition law on $E$, we get $2(0,-i) = (0,i)$, $2(0,i) = (0,-i)$, and $-(0,i)=(0,-i)$.  In other words, $(0,\pm i)$ are $3$-torsion points of $E$.  Using these facts, we can observe a contradiction in each case where a branch point occurs in more that one entry of $\omega$.  
    
    Case 1. The degree $m$ is odd, so the branch points of $\rho_m$ are $(0,i)$ and $(0,-i)$. Suppose $P+Q_l$ is a branch point.  If $P+Q_l=P+Q_j$, then we have a contradiction because $Q_1,\ldots ,Q_r$ are distinct.  If $P+Q_l=-(P+Q_j)$, then $Q_l-Q_j = P+Q_l - (P+Q_j)= 2(P+Q_l)= 2(0,\pm i) = (0, \pm i)$ which contradicts the hypothesis.

    Case 2. The degree $m$ is even, so the branch points of $\rho_m$ are $(0,i)$, $(0,-i)$, and the point at infinity $\mathcal{O}$.  Suppose $P+Q_l$ is a branch point.  The cases where $P+Q_l$ and $P+Q_j$ are both $(0,\pm i)$ proceed as above, so let $P+Q_l = \mathcal{O}$.
    
    If $P + Q_j = (0, \pm i)$ then $P+Q_j - (P + Q_l) = (0, \pm i)$. This is a contradiction with the hypothesis.  If $P+Q_j = \mathcal{O}$, then we have $Q_j = Q_l$ which is also a contradiction.
    
    So, at every branch point of $\rho^r_m$, exactly one of the coordinates is a branch point of $\rho_m$.

\end{proof}

Recall, the value of $r$ depends on $\overline{g}$ and $g$.

\begin{lem} \label{numbranchpts}
    Suppose $Q_j - Q_l \not = (0,\pm i)$ for all $1 \leq j,l \leq r$. Let B be the set of branch points of the cover $\rho_m^r: T_m \longrightarrow W$.  Then for all $g > 2\overline{g}-1$,
    $$ \#B =  \begin{cases}
2r=4(g-3), & \text{ if } \overline{g} = 2. \\ 
3r=6(g+1-2\overline{g}), & \text{ if } \overline{g} > 2 \text{ and } \overline{g} \not \equiv 0 \bmod 3.\\ 
\end{cases}$$
\end{lem}

\begin{proof}
    
    In the case $\overline{g}=2$, the branch points of $\rho_2$ are $\{ (0,\pm i)\}$. Define $$B_2 := \{ \omega \in W \mid P+Q_j = (0, \pm i), \text{ for one index } j \}.$$  By Lemma \ref{TWbranch}, $B_2$ is the set of branch points of $\rho_2^r: T_2 \longrightarrow W$.  Notice there are $r$ choices for $j$, so $\mid B_2 \mid = 2r$.

    If $\overline{g}>2$, the branch points of $\rho_m$ (for $m>2$) are $\{ (0,\pm i), \mathcal{O}\}$. Let $$B_{\overline{g}>2} := \{ \omega \in W \mid P+Q_j = (0, \pm i) \text{ or } \mathcal{O}, \text{ for one index } j \}.$$  

    Similarly, by Lemma \ref{TWbranch} $B_{\overline{g}>2}$ is the set of branch points of $\rho_m^r: T_m \longrightarrow W$ for $m > 2$.  There are $r$ choices for $j$, so there are $3r$ branch points.  Since $r= 2(g+1-2\overline{g})$ we obtain the result.
\end{proof}

We again refer to the ramification index of a branch point of $\rho_m^r: T_m \longrightarrow W$ since this cover is also abelian by Lemma \ref{TWdeggal}.

\begin{lem} \label{TWDiff}
    Let $R$ be the set of ramification points in $T_m$ under the map $\rho_m^r: T_m \longrightarrow W$ and let $e_t$ be the ramification index for $t \in R$ .  Then, 
    $$D := \sum_{t \in R}(e_t -1) = \begin{cases}
    r2^{r}, & \text{ if } \overline{g} = 2. \\ 
    r(2\overline{g}+1)/(3)^{r-1}(2\overline{g}-2), & \text{ if } \overline{g} >2  \text{ and } \overline{g} \equiv 1 \bmod 3. \\ 
    r((2/3)(\overline{g}+1))^{r-1}(2\overline{g}-2), & \text{ if } \overline{g} > 2  \text{ and } \overline{g} \equiv 2 \bmod 3. \\
    \end{cases}$$
\end{lem}

\begin{proof} 
Let $B$ be the set of branch points of $\rho^r_m$ and $\omega \in B$.  By Lemma \ref{TWdeggal}, $(\rho^r_m)^{-1}(B) = R$.

Using Lemma \ref{TWbranch}, if $\omega$ is a branch point of $\rho_m^r$ then there is an index $0 \leq l \leq r$ so that $P+Q_l$ is a branch point of $\rho_m$ and no other entry of $\omega$ is a branch point of $\rho_m$.  Then the ramification indices of $\rho_m^r$ above $\omega$ and $\rho_m$ above $P+Q_l$ are equal, so we use Lemma \ref{ramindexXm} to find the ramification index above $\omega$.  Then there are $m^{r-1}(m/e)$ points in the fiber of $\rho_m^r$ above $\omega$ with the value of $e$ in Lemma \ref{ramindexXm}.

Case 1. If $\overline{g}=2$, then $m=2$ and there are $m^{r-1}$ points in $(\rho^r_m)^{-1}(\omega)$ for each choice of $\omega$ and the ramification indices of those points are all equal to $2$ by the fact above and Lemma \ref{ramindexXm}, so $\sum_{t \in (\rho^r_2)^{-1}(\omega)} (e_t -1) = 2^{r-1}$.  Then by Lemma \ref{numbranchpts}, there are $2r$ choices of $\omega$.

Case 2. If $\overline{g}>2$ and $\overline{g} \equiv 1 \bmod 3$, there are $m^{r-1}$ points in $(\rho^r_m)^{-1}(\omega)$ for each choice of $\omega$.  The ramification index of each point in $(\rho^r_m)^{-1}(\omega)$ is $m=(2\overline{g}+1)/3$. So, for all $\omega \in B$, $\sum_{t \in (\rho^r_m)^{-1}(\omega)}(e_t -1) =m^{r-1}(m-1)$.  By Lemma \ref{numbranchpts}, there are $3r$ choices of $\omega$.
  
Case 3.  If $\overline{g}>2$ and $\overline{g} \equiv 2 \bmod 3$, we have two cases.  Again let $\omega=(P,P+Q_2,\ldots ,P+Q_r) \in B$.

Case 3a. If $P+Q_l = (0, \pm i)$, as in Case 2, there are $m^{r-1}$ points in $(\rho^r_m)^{-1}(\omega)$ and each ramification index is $m=(2/3)(\overline{g}+1)$.  Here there are $2r$ choices for $\omega$.

Case 3b. If $P+Q_l = \mathcal{O}$, there are $m^r/(m/2)= 2m^{r-1}$ points in $(\rho^r_m)^{-1}(\omega)$.  So, each ramification point in the pre-image has a ramification index of $m/2$. So we obtain $\sum_{t \in (\rho^r_m)^{-1}(\omega)}(e_t -1) = 2m^{r-1}(m/2 -1)= m^{r-1}(m-2)$.
Then there are $r$ choices of $\omega$ in this case.
\end{proof}

One can rewrite Lemma \ref{TWDiff} in terms of $g$ and $\overline{g}$ instead of $r$ and $\overline{g}$ if desired.

\subsection{Genus of $T_m$}

By Lemma \ref{TWdeggal}, we index the branches of the cover $T_m \longrightarrow W$ over a chosen base point by elements of $(\Z/m\Z)^r$.

\begin{prop} \label{Tcon}
        Suppose $Q_j - Q_l \not = (0,\pm i)$ for all $ 1 \leq j,l \leq r$.  Then $T_m$ is connected.
\end{prop}

    \begin{proof}
    Note, it suffices to prove $T_m$ is connected over $\C$.
    
    Consider the monodromy of the cover $\rho_m^r: T_m \longrightarrow W$.  To show $T_m$ is connected, we need to show that the action of lifting loops in $W$ to paths in $T_m$ induces a transitive action on the branches of $T_m$.  First, note that $W \cong E$, since every point in $W$ is determined by one point, $P$ in $E$.  From Lemma \ref{TWdeggal}, each branch of $T$ is indexed by an element of $(\Z/m\Z)^r$.  
    
    Consider a loop $L$ in $E$ with base point $\mathcal{O}$.  The loop $L$ in $E$ determines a loop $\Tilde{L}$, in $T_m$ by letting $P$ travel around $L$ and defining points $(P, P+Q_2,...,P+Q_r)$.  
    
    Fix $j$.  From $\mathcal{O}$ let $L$ circle counterclockwise around the point $(0,i)-Q_j$ (we use $(0,i)$ without loss of generality, the argument goes through for $(0,-i)$ as well) and return to $\mathcal{O}$. Notice, if $L$ travels around $(0,i)-Q_j$ then in $\Tilde{L}$ the $j$th entry travels around $(0,i)$.  We can understand $\Tilde{L}$ as a loop where the $j$th entry circles a point $\omega \in B$ as defined above where the $jth$ entry of $\omega$ is a branch point of $\rho_m$.

    By Lemma \ref{TWbranch}, for indices $l \not = j$, the $l$th component of $\omega$ is not a branch point of $\rho_m$.  This means that we can shrink $\Tilde{L}$ within its homotopy class so that only one entry travels around $(0, i)$.

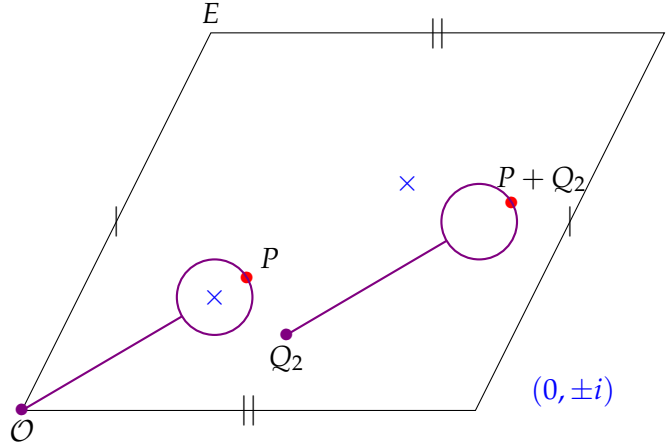
\begin{figure}[h]
\begin{tikzpicture}[xslant=0,xscale=1,yscale=1, every node/.append style={inner xsep=0pt}]

\begin{scope}[xslant=0.5,xscale=6,yscale=5]
    \draw (0,0) rectangle (1,1);
    \draw(.5,0) node {$||$};
    \draw(.5,1) node {$||$};
    \draw(0,.5) node {$|$};
    \draw(1,.5) node {$|$};
    \draw(.3,.3) node {\color{blue}{$\times$}};
    \draw(.6,.6) node {\color{blue}{$\times$}};
    \draw (0,1) node[above] {$E$};
    \draw[violet,thick] (0,0) node {$\bullet$} to (.25,.25);
    \draw (0,0) node[below] {$\mathcal{O}$};
    \draw[violet,thick] (.5,.2) node {$\bullet$} to (.75, .45);
    \draw (.5,.2) node[below] {$Q_2$};
    \draw(.35,.35) node {\color{red}{$\bullet$}};
    \draw(.35,.35) node [above]{\hspace{5mm} $P$};
    \draw(.85,.55) node {\color{red}{$\bullet$}};
    \draw(.85,.55) node [above]{\hspace{7mm} $P+Q_2$};
    \draw (1.1,.15) node[right] {};
    \draw (1.1,.05) node[right] {\color{blue}{$(0, \pm i)$}};
\end{scope}
    \draw[violet,thick] (2.55,1.5) circle [radius=0.5];
    \draw[violet,thick] (6.05, 2.5) circle [radius=0.5];
\end{tikzpicture}
\caption[Loops on curve $E$ as $P$ varies.]{A diagram of loops on $E$ for $r=2$ as $P$ varies, avoiding more than one branch point of $\rho_2$.}
    \label{fig:loopsonE}
\end{figure}  

When $\Tilde{L}$ is lifted up to $T_m$ it creates a path starting at a point on the branch labeled by $\vec{0}$ and ending at a point on the branch labeled by $e_j$, the $j$th generator of $(\Z/m\Z)^r$.  This is because in $\Tilde{L}$, the $j$th entry traveled around a branch point of $\rho_m$ and no other entries traveled around a branch point.  Since $j$ is an arbitrary index, the image of loops generate the Galois group $(\Z/m\Z)^r$.  Then $T_m$ is connected.
\end{proof}

Now that we've shown $T_m$ is connected we can use the Riemann-Hurwitz formula to find the genus of $T_m$.  

\begin{thm} \label{main}
Suppose $Q_j-Q_l \not = (0,\pm i)$ for $1 \leq j, l \leq r$.  Let $\overline{g}$ the genus of $X_m$ and $r=2(g+1-2\overline{g})$. Then for all $g>2\overline{g}-1$ there exists a complete curve as the image of $T_m \longrightarrow \calM_{g,\overline{g}}$ where $T_m$ has genus

$$G_T =  \begin{cases} 
r2^{r-1}+1, & \text{ if } \overline{g} = 2. \\ 
r(\frac{2\overline{g}+1}{3})^{r-1}(\overline{g}-1)+1, & \text{ if } \overline{g}>2 \text{ and } \overline{g} \equiv 1 \bmod 3. \\ 
r(\frac{2}{3}(\overline{g}+1))^{r-1}(\overline{g}-1)+1, & \text{ if } \overline{g} > 2 \text{ and } \overline{g} \equiv 2 \bmod 3. \\
    \end{cases}$$
\end{thm}

    \begin{proof}
        By Proposition \ref{Tcon}, $T_m$ is connected.  By Lemma \ref{TWDiff} we know the degree $D$ of the ramification divisor.  In all cases, 
        $$2G_T-2 = m^r(2G_E-2)+D$$
        $$G_T= 1+ \frac{1}{2}D$$
Then we can simplify to get the result.
    \end{proof}

\section{An Alternate Choice for Covers $X_m$ over $E$} \label{sect:altconst}

Instead of the cover in Sections \ref{sect:Xmconst} and \ref{sect:genusT}, let $X_m$ be the smooth projective curve with affine equation $ y_2^{2m}-(x_2^3 - 1)=0$.  Let $E$ be defined as above. Then we define a new cover for $m > 1$,
\begin{equation} \label{YmoverE}
\begin{tikzcd}
 X_m: y_2^{2m}-(x_2^3-1)= 0 \dar{\rho_m} \\
 E: y_1^2-(x_1^3-1) =0
 \end{tikzcd}
\end{equation}

 such that $\rho_m(x_2,y_2) = (x_2,y_2^m)$.  The cover is Galois with Galois group $\Z/ m \Z$.  We proceed similarly to Sections \ref{sect:Xmconst} and \ref{sect:genusT}.

\begin{lem} \label{altXmgenus}
    The curve $X_m$ has genus 
    \[\overline{g}= \begin{cases} 
      2m-1 & \text{ if } m \not \equiv 0 \bmod 3. \\
      2m-2 & \text{ if } m \equiv 0 \bmod 3.
   \end{cases}
\]
\end{lem}

\begin{proof}
       To calculate the genus, we use the Riemann-Hurwitz Formula on the cover $X_m \longrightarrow \mathbb{P}^1$ where $(x_2,y_2) \mapsto x_2$.  This map has degree $2m$, and at branch points $b_0 = (\zeta_3, 0), b_2 = (\zeta_3^2, 0),$ and $b_3=(1,0)$ the ramification indices are all $2m$.  So the inertia type is $\vec{a} = (1,1,1,2m-3)$ with the last branch point being the point at infinity.  The number of points above the branch point $b_i$ is $\gcd(2m,a_i)$, so we have ramification index $2m$ for all except $b_\infty$ which has \[e_\infty = \begin{cases} 
      2m & \text{ if } m \not \equiv 0 \bmod 3. \\
      \frac{2m}{3} & \text{ if } m \equiv 0 \bmod 3. \end{cases} \]

We have two cases for Riemann-Hurwitz.

Case 1. If $m \not \equiv 0 \bmod 3$, then $\overline{g} = 2m-1$.

Case 2.  If $m \equiv 0 \bmod 3$, then $\overline{g} = 2m-2$.

\end{proof}

Note: Case 1 gives $\overline{g} \equiv 1, 3 \bmod 6$ and case 2 gives $\overline{g} \equiv 4 \bmod 6$.

\begin{cor}
    If $m \equiv 0 \bmod 3 $ then $X_m$ has $3$ points at infinity.  Otherwise $X_m$ has one point at infinity.
\end{cor}

\begin{proof}
    Follows from the ramification indices in Lemma \ref{altXmgenus}.
\end{proof}

The Galois part of Lemma \ref{ghatmgenusandgal} applies for the new choice of $X_m$ and $\rho_m$.  The cover $\rho_m$ is Galois with Galois group $\mathrm{Gal}(X_m/E) \equiv \Z/m\Z$ because the associated function field extension is given by adjoining an $m$th root of $y_1$.

\begin{cor}
    For a fixed $m >2$ and with the choice of $\rho_m: X_m \longrightarrow E$ defined above there exists an explicit complete curve in $\calM_{g,\overline{g}}$ for all $\overline{g} \geq 2$ such that $\overline{g} \equiv 1, 3, 4 \bmod 6$ and $g > 2\overline{g} -1$.
\end{cor}

\begin{proof}
    Proof proceeds similarly to the proof of Corollary \ref{rho_mCC}. When $m \not \equiv 0 \bmod 3$, $\overline{g} = 2(3k+1)-1= 6k+1$ or $\overline{g} = 2(3k+2)-1= 6k +3$ and when $m \equiv 0 \bmod 3$, $\overline{g} = 2(3k)-2= 6k-2 \equiv 4 \bmod 6 $.
\end{proof}

Define $T$ as in Section \ref{sect:genusT}. Lemma \ref{TWdeggal} applies to our new cover.  Then we mimic Section \ref{sect:genusT} to find the genus of $T_m$ in this case.    

\begin{lem} \label{numbranchptsalt}
    Suppose $Q_j-Q_l \not = (\zeta_3^n, 0)$ for all $1 \leq j,l \leq r$ and $0 \leq n \leq 2$. Let B be the set of branch points of the cover $\rho_m^r: T_m \longrightarrow W$.  Then for all $g > 2\overline{g}-1$,
$$ \#B =  \begin{cases}
    3r, & \text{ if } m = 3.\\
    4r, & \text{ otherwise}. \\ 
    \end{cases}$$
\end{lem}

\begin{proof}
    The branch points of $\rho_m: X_m \longrightarrow E$ are $(1,0),(\zeta_3, 0),(\zeta_3^2,0)$ and possibly $\mathcal{O}$.  A point $\omega \in W$ is a branch point of the cover $\rho_m^r: T_m \longrightarrow W$ if and only if exactly one entry of $\omega$ is a branch point of $\rho_m$.  Similarly to the previous section, we assume there are two entries of $\omega$, $P + Q_j$ and $P+Q_l$, that are $\mathcal{O}$ or $(\zeta_3^n, 0)$ where $0 \leq n \leq 2$.  If $P + Q_j = P+Q_l$ we contradict the condition that $Q$'s are distinct.  If $P + Q_j \not = P+Q_l$ and $P + Q_j = (\zeta_3^{n_1},0), P+Q_l = (\zeta_3^{n_2},0)$, then $(P + Q_j)-(P+Q_l) = (\zeta_3^{n_3},0)$ the colinear point where $n_3 \not = n_1,n_2$, since these are the two torsion points.  Finally, if $P + Q_j \not = P+Q_l$ and $P + Q_j = \mathcal{O}, P+Q_l = (\zeta_3^{n_1},0)$ then $(P + Q_j)-(P+Q_l) = (\zeta_3^{n_1},0)$.
    
\end{proof}

\begin{lem} \label{ramindexalt}
    Let $L$ be a branch point of $\rho_m: X_m \longrightarrow E$ with ramification index $e$.  The value of $e$ is
    $$e = \begin{cases}
    m, & \text{ if } P+Q_l = (1,0), (\zeta_3,0) \text{ or } (\zeta_3^2,0). \\ 
    m, & \text{ if } m \not \equiv 0 \bmod 3,   \text{ and }  P+Q_l = \mathcal{O}. \\
    m/3, & \text{ if } m \equiv 0 \bmod 3,   \text{ and }  P+Q_l = \mathcal{O}. \\
    \end{cases}$$
\end{lem}

\begin{proof}
    A branch point $L$ is an element of the set $\{\mathcal{O}, (1,0), (\zeta_3,0), (\zeta_3^2,0)\}$.

    Case 1. If $L = (\zeta_3^n,0)$, then $\rho_m^{-1}(L)=(\zeta_3^n,0)$.  Usually there are $m$ pre-images, so $e=m$.

    Case 2. If $m \not \equiv 0 \bmod 3$ and $L = \mathcal{O}$, then $X_m$ has one point at infinity which is the pre-image of $\mathcal{O}$.  So, $e=m$.

    Case 3. If $m \equiv 0 \bmod 3$ and $L = \mathcal{O}$, then $X_m$ has three points at infinity which makes up the pre-image of $\mathcal{O}$.  Since $\rho_m$ is a Galois cover, $e=m/3$ for each point in the pre-image.
\end{proof}

Just as in Section \ref{sect:genusT} if $P+Q_l$ is a branch point of $\rho_m$ then $\omega=(P,P+Q_2,\ldots ,P+Q_r)$ is a branch point of $\rho_m^r$ and the ramification index of $\omega$ is $e$ as in Lemma \ref{ramindexalt}.

\begin{lem} \label{TWDiffalt}
    Let $R$ be the set of ramification points in $T_m$ under the map $\rho_m^r: T_m \longrightarrow W$ and let $e_t$ be the ramification index for $t \in R$ .  Then, 
    $$D := \sum_{t \in R}(e_t -1) = \begin{cases}
    2r3^{r} & \text{ if } \overline{g} = 5.   \\
    4rm^{r-1}(m-1), & \text{ if } \overline{g} \equiv 1, 3 \bmod 6 \text{ and } \overline{g} \not = 5. \\ 
    3rm^{r-1}(m-1) + rm^{r-1}(m-3), & \text{ if } \overline{g} \equiv 4 \bmod 6. \\ 
    \end{cases}$$
\end{lem}

\begin{proof}
    As in Lemma \ref{TWDiff}, we can use the number of branch points in Lemma \ref{numbranchptsalt} and the ramification indices from Lemma \ref{ramindexalt} to calculate $D$ in each case.  There are $m^{r-1}(m/e)$ points in the fiber of $\rho_m^r$ above a branch point, then to obtain the result by finding $m^r - m^{r-1}(m/e)$ and multiplying by the number of branch points with that ramification behavior.
\end{proof}

\begin{thm} \label{mainalt}
Suppose $Q_j-Q_l \not = (\zeta_3^n, 0)$ for $1 \leq j, l \leq r$ and $0 \leq n \leq 2$.  Let $\overline{g}$ the genus of $X_m$, and $r=2(g+1-2\overline{g})$. Then, for all $g>2\overline{g}-1$ there exists a complete curve as the image of $T_m \longrightarrow \calM_{g,\overline{g}}$ where $T_m$ has genus

$$G_T =  \begin{cases} 
r3^r + 1, & \text{ if } \overline{g} = 5 \\
r(\frac{\overline{g}+1}{2})^{r-1}(\overline{g}-1) +1, & \text{ if } \overline{g} \not = 5 \text{ and } \overline{g} \equiv 1, 3 \bmod 6. \\ 
\frac{3}{2}r(\frac{\overline{g}+2}{2})^{r-1}(\frac{\overline{g}}{2}) + \frac{1}{2}r(\frac{\overline{g}+2}{2})^{r-1}(\frac{\overline{g}+4}{2}) +1, & \text{ if } \overline{g} \equiv 4 \bmod 6. \\ 
\end{cases}$$
\end{thm}

\begin{proof}
    By Proposition \ref{Tcon} holds for our new cover $T_m \longrightarrow W$ because we have the same Galois group and monodromy behavior, so $T_m$ is connected.  By Lemma \ref{TWDiffalt} we know the degree $D$ of the ramification divisor.    The value of $m=(\overline{g}+1)/2$ or $m=(\overline{g}+2)/2$ depending on if $3 \mid m$. Then we once again, by the Riemann Hurwitz formula
        $$G_T=\frac{1}{2}D + 1$$
\end{proof}

If desired, one can again rewrite in terms of $g$ and $\overline{g}$ since the value of $r$ is $2(g+1-2\overline{g})$.

\bibliographystyle{alpha}
\bibliography{sources}

@article{GonDiezHarvey1991,
  title={On complete curves in moduli space {I}},
  author={González Díez, Gabino and Harvey, William J.},
  journal={Mathematical Proceedings of the Cambridge Philosophical Society, vol. 110, no. 3},
  pages={461-466},
  year={1991},
  publisher={}
}

@article{Zaal1995,
  title={Explicit complete curves in the moduli space of curves of genus three},
  author={Zaal, Chris},
  journal={Geom Dedicata 56},
  pages={185-196},
  year={1995},
  publisher={}
}

@article{GonDiezHarveyII,
  title={On complete curves in moduli space {II}},
  author={González Díez, Gabino and Harvey, William J.},
  journal={Mathematical Proceedings of the Cambridge Philosophical Society, vol. 110, no. 3},
  pages={467-472},
  year={1991},
  publisher={}
}

@article{Harris1983,
  title={Recent Work on $\mathcal{M}_g$},
  author={Harris, Joe},
  month={August},
  journal={Proceedings of International Congress of Mathematicians, Warsaw},
  pages={719-726},
  year={1983},
  publisher={PWN-Polish Scientific Publishers}
}

@misc{Choi2023,
      title={Complete Subvarieties of $\rm{M}_{g,n}$ and a Lifting Problem}, 
      author={Choi, Daebeom},
      year={2023},
      eprint={2304.08568},
      archivePrefix={arXiv},
      primaryClass={math.AG},
      url={https://arxiv.org/abs/2304.08568}, 
}

@article{Diaz1984,
  title={A Bound on the Dimensions of Complete Subvarieties of Complete Subvarieties of $\mathcal{M}_g$},
  author={Diaz, Steven},
  month={June},
  journal={Duke Mathematical Journal, vol. 51},
  pages={405-408},
  year={1984},
  publisher={Duke University Press}
}

@thesis{Zaal2005,
  title={Complete subvarieties of moduli spaces of algebraic curves},
  author={Chris Zaal},
  school={Universiteit van Amsterdam},
  year={2005}
}

@article{Looijenga1995,
  title={On the tautological ring of $\mathcal{M}_g$},
  author={Looijenga, Eduard},
  journal={Inventiones mathematicae},
  year={1995},
  volume={121},
  pages={411-419},
  url={https://doi.org/10.1007/BF01884306}
}

@article{Chen2010,
url = {https://doi.org/10.1515/crelle.2010.092},
title = {Covers of elliptic curves and the moduli space of stable curves},
author = {Dawei Chen},
pages = {167--205},
volume = {2010},
number = {649},
journal = {Journal für die reine und angewandte Mathematik},
doi = {doi:10.1515/crelle.2010.092},
year = {2010}
}

@article{Shaska2006,
url = {http://eudml.org/doc/281475},
title = {Subvarieties of the hyperelliptic moduli determined by group actions},
author = {Tanush Shaska},
pages = {355--374},
volume = {32},
number = {4},
journal = {Serdica Mathematical Journal},
publisher = {Institute of Mathematics and Informatics Bulgarian Academy of Sciences},
year = {2006},
}

@article{Hidalgo2024,
url = {https://doi.org/10.1007/s00031-024-09870-3},
title = {On non-normal subvarieties of the moduli space of riemann surfaces},
author = {Ruben A. Hidalgo and Jennifer Paulhus and Sebastian Reyes-Carocca and Anita M. Rojas},
pages = {675--704},
volume = {31},
journal = {Transformation Groups},
doi = {doi:10.1007/s00031-024-09870-3},
year ={2026}
}

@article{Shaska2004,
author = {Shaska, Tanush},
title = {Some special families of hyperelliptic curves},
journal = {Journal of Algebra and Its Applications},
volume = {03},
number = {01},
pages = {75-89},
year = {2004},
doi = {10.1142/S0219498804000745},
URL = {https://doi.org/10.1142/S0219498804000745},
eprint = {https://doi.org/10.1142/S0219498804000745
}}

\end{document}